\documentclass[twocolumn]{article}
\usepackage[utf8]{inputenc}
\usepackage[left=2cm, right=2cm, top=2cm]{geometry} 
\usepackage[small]{titlesec}
\usepackage{etoolbox}
\usepackage{amsthm}
\makeatletter
\patchcmd{\ttlh@hang}{\parindent\z@}{\parindent\z@\leavevmode}{}{}
\patchcmd{\ttlh@hang}{\noindent}{}{}{}
\makeatother

\newcommand{\minimize}[3]{\begin{array}{rl}
		{\underset{#2}{\textrm{minimize}}} & \begin{aligned}[t] 
			#1
		\end{aligned}  \\[15pt]
		\textrm{subject to} &
		\begin{aligned}[t] 
			#3
		\end{aligned}
	\end{array}}

	\usepackage{mathptmx} % assumes new font selection scheme installed
	\usepackage{times} % assumes new font selection scheme installed
	\usepackage{amsmath} % assumes amsmath package installed
	\usepackage{amssymb}  % assumes amsmath package installed
	\usepackage{algorithm}
	\usepackage{algpseudocode}
	\usepackage{color}
	\usepackage{bm}
	\usepackage{makecell}
	\usepackage{subfig, float}
	\usepackage{tikz,pgfplots}
	\usepackage{epstopdf}
	\usepackage{natbib}
	\usetikzlibrary{external,arrows,shapes,positioning,plotmarks, arrows.meta}
	\usepgfplotslibrary{external} 
	\newlength\figureheight 
	\newlength\figurewidth  
	
	\newtheorem{myassum}{Assumption}
	\newtheorem{mythm}{Theorem}
	\newtheorem{mycor}{Corollary}
	\newtheorem{myrem}{Remark}
	\newtheorem{mylem}{Lemma}
	
	\DeclareMathAlphabet{\mathcalOld}{OMS}{cmsy}{m}{n}

	\newcommand{\tfinal}{t_{f}}
	\newcommand{\tfinalnom}{\bar{t}_{f}}
	\newcommand{\tinitial}{t_{0}}
	\newcommand{\tstar}{T^{\star}}
	\newcommand{\xfinal}{\bm{x}_{f}}
	\newcommand{\xinitial}{\bm{x}_{0}}
	
	\newcommand{\xcurr}{\bm{x}_{\mathrm{cur}}}
	\newcommand{\tcurr}{t_k}
	\newcommand{\ustar}{\bm{u}^{\star}}
	\newcommand{\taustar}{\tau^{\star}}

\title{\LARGE \bf
An Optimization-Based Receding Horizon Trajectory Planning Algorithm 
}

\date{}
\author{Kristoffer Bergman, Oskar Ljungqvist, Torkel Glad and Daniel Axehill}% <-this % stops a space

\begin{document}
	
	% declaration of the new block
	\algblock{ParFor}{EndParFor}
	% customising the new block
	\algnewcommand\algorithmicparfor{\textbf{parfor}}
	\algnewcommand\algorithmicpardo{\textbf{do}}
	\algnewcommand\algorithmicendparfor{\textbf{end\ parfor}}
	\algrenewtext{ParFor}[1]{\algorithmicparfor\ #1\ \algorithmicpardo}
	\algrenewtext{EndParFor}{\algorithmicendparfor}

	\maketitle
	\thispagestyle{empty}
	\pagestyle{empty}

	%%%%%%%%%%%%%%%%%%%%%%%%%%%%%%%%%%%%%%%%%%%%%%%%%%%%%%%%%%%%%%%%%%%%%%%%%%%%%%%%		
		\textbf{\textit{Abstract ---}}\textbf{This paper presents an optimization-based receding horizon trajectory planning algorithm for dynamical systems operating in unstructured and cluttered environments. The proposed approach is a two-step procedure that uses a motion planning algorithm in a first step to efficiently find a feasible, but possibly suboptimal, nominal solution to the trajectory planning problem where in particular the combinatorial aspects of the problem are solved. The resulting nominal trajectory is then improved in a second optimization-based receding horizon planning step which performs local trajectory refinement over a sliding time window. In the second step, the nominal trajectory is used in a novel way to both represent a terminal manifold and obtain an upper bound on the cost-to-go online. This enables the possibility to provide theoretical guarantees in terms of recursive feasibility, objective function value, and convergence to the desired terminal state. The established theoretical guarantees and the performance of the proposed algorithm are verified in a set of challenging trajectory planning scenarios for a truck and trailer system.   }
	
\section{Introduction} \label{sec:intro}
In recent decades, an extensive amount of research has been conducted in the area of motion planning for autonomous vehicles~\citep{lavalle2006planning, paden2016survey}. However, the problem of computing locally optimal trajectories for dynamical systems in confined and unstructured environments is still considered as a difficult task. In this paper, the optimal motion planning problem is defined as the problem of finding a feasible and collision-free trajectory that brings the system from its initial state to a desired terminal state while a performance measure is minimized. The computed trajectory is then intended to be used as reference to a trajectory tracking or path following controller~\citep{LjungqvistCDC2018,paden2016survey,ljungqvist2019path}.
%Goal state tycker jag vi ska säga så vi inte förväxlar med terminal manifold alt. terminal state for RHOP?

The optimal motion planning problem is in general hard to solve by directly applying optimal control techniques, since the problem in general is nonconvex due to obstacle-imposed constraints and nonlinear system dynamics. Therefore, approximate methods in terms of motion planning algorithms are commonly used~\citep{lavalle2006planning}. One commonly used approach for dynamical systems is to apply sampling-based planners, which are either based on random or deterministic exploration of the vehicle's state space~\citep{lavalle2006planning}. One approach based on random sampling is RRT$^{\star}$ which is a popular motion planning algorithm for dynamical systems where an efficient steering function is available~\citep{karaman2013sampling,banzhaf2018g}. %In these methods, the orginal {RRT} algorithm is extended with a rewiring step in each expansion of the search tree leading to an algorithm that shown to be to be asymptotically optimal~\citep{karaman2013sampling}. 
Unless an efficient steering function is available, the RRT$^{\star}$ algorithm becomes computationally inefficient as multiple optimal control problems (OCPs) have to be solved online at each tree expansion~\citep{stoneman2014embedding}. 

A popular deterministic sampling-based motion planner is the lattice-based motion planner, which uses a finite set of precomputed motion segments, or motion primitives, online to find an optimal solution to a discretized version of the motion planning problem~\citep{pivtoraiko2009differentially}. A benefit with this method is that efficient graph-search algorithms can be used online such as A$^{\star}$~\citep{hart1968formal}, making it real-time applicable~\citep{pivtoraiko2009differentially, ljungqvist2019path}. 
However, since the lattice-based planner uses a discretized search space, 
the computed solution can be noticeably suboptimal and a latter post-optimization step is often desirable to use~\citep{dolgov2010path,andreasson2015fast}. A related technique is proposed in our previous work in~\citep{bergman2019bimproved}, where an optimization-based improvement step is added, aiming at locally improving the solution from a lattice-based planner without being limited to a discrete search space. Compared to previous work, a tight integration between the motion planner and the optimization step was introduced. This new approach was shown to have significant benefits over existing related methods in terms of solution quality and reliability.
However, the introduced improvement step increases the motion planner's latency time and hence, the time before the trajectory can start being executed. To reduce the computation time of the improvement step, and thus enable a faster start of the execution phase, a receding horizon trajectory planning approach is proposed in this paper where the nominal trajectory from the motion planning algorithm is improved iteratively during the execution phase. 

Optimization-based receding horizon planning (RHP) is commonly used in on-road applications, where the structure of the road environment is utilized to evaluate several candidates with different terminal states centered around the vehicle's lane. In~\citep{werling2012optimal}, these candidates are efficiently computed using quintic polynomials. In unstructured environments, optimization-based RHP has mainly been applied on unmanned areal vehicles (UAVs)~\citep{schouwenaars2004receding,kuwata2005robust, liu2017planning}. The RHP approach is motivated in many applications due to limited sensing range, which makes it unnecessary to optimize the full horizon trajectory to the terminal state~\citep{liu2017planning}. Common for these methods are that outside the vehicle's planning range, a geometric planning algorithm is used to compute a simplified trajectory to the goal, e.g., a shortest distance trajectory that avoids known obstacles but disregards the system dynamics. The simplified trajectory is then used to estimate the cost-to-go, which enables a trade-off between short term and long term trajectory selection. This technique has been shown to work well for agile systems such as quadcopters. However, for systems that are less agile (such as truck and trailer systems), using, e.g., a geometric algorithm to estimate the cost-to-go can in worst case lead to infeasibility~\citep{pivtoraiko2009differentially,bergman2019bimproved}. 

To avoid potential infeasibility caused by using a simplified cost-to-go estimate when solving the RHP problem, the main contribution in this work is to use a nominal trajectory computed by a motion planning algorithm in a novel way to define a terminal manifold and an upper bound on the optimal cost-to-go. This result is utilized to provide theoretical guarantees on feasibility during the entire planning horizon, objective function value improvement and convergence to the terminal state. These theoretical results are used to define a practical RHP algorithm, whose performance is verified in a number of challenging motion planning problems for a truck and trailer system.

The remainder of the paper is organized as follows. The optimal motion planning problem is posed in Section~\ref{sec:prob}. In Section~\ref{sec:rhi}, the RHP problem is defined and theoretical guarantees presented. These results are used in Section~\ref{sec:alg} to present an algorithm to iteratively improve the nominal trajectory using RHP. A simulation study for a truck and trailer system is presented in Section~\ref{sec:Res}, followed by conclusions and future work in Section~\ref{sec:conc}.    

\section{Problem formulation} \label{sec:prob}
In this paper, continuous-time nonlinear systems in the form
\begin{equation} \label{eq:system}
\dot{\bm{x}}(t) = f(\bm{x}(t), \bm{u}(t)), \quad \bm{x}(\tinitial) = \xinitial,
\end{equation} 
are considered, where $\bm{x} \in \mathbf{R}^n$ and $\bm{u} \in \mathbf{R}^m$ denote the state and control signal of the system, respectively.  These are subject to the following constraints:
\begin{equation}
\bm{x} \in \mathcalOld{X} \subseteq \mathbf{R}^n, \quad \bm{u} \in \mathcalOld{U} \subseteq \mathbf{R}^m.
\end{equation}
Furthermore, the system should not collide with obstacles, where the obstacle region is defined as \mbox{$\mathcalOld{X}_{\text{obst}} \subset \mathbf{R}^n$}. Thus, in motion planning problems, the state space is constrained as:
\begin{equation} \label{eq:obst_av}
\bm{x} \in \mathcalOld{X}_{\text{free}} = \mathcalOld{X} \setminus \mathcalOld{X}_{\text{obst}}.
\end{equation}
This constraint is in general non-convex since $\mathcalOld{X}_{\text{free}}$ is defined as the complement set of $\mathcalOld{X}_{\text{obst}}$. 

The motion planning problem can now be defined as the problem of computing a feasible (i.e. satisfying \eqref{eq:system}-\eqref{eq:obst_av}) state and control signal trajectory $(\bm{x}(\cdot), \bm{u}(\cdot) )$  that moves the system from $\xinitial \in \mathcalOld{X}_{\text{free}}$ to a desired terminal state, $\xfinal \in \mathcalOld{X}_{\text{free}}$, while a performance measure~$J_{\text{tot}}$ is minimized. This problem can be posed as a continuous-time OCP:
\begin{equation}
\minimize{J_{\mathrm{tot}}(\xinitial,\bm{u}(\cdot)) = \int_{\tinitial}^{\tfinal} \ell (\bm{x}(t), \bm{u}(t)) \mathrm{d}t}{\bm{u}(\cdot), \;\tfinal}{&\bm{x}(\tinitial) = \xinitial,
	\quad \bm{x}(\tfinal) = \xfinal,  \\ &\dot{\bm{x}} (t) = f(\bm{x}(t),\bm{u}(t)),  \\ &\bm{x}(t) \in \mathcalOld{X}_{\mathrm{free}}, \; \bm{u}(t) \in \mathcalOld{U}&&\hspace*{-2ex}t \in [\tinitial,\tfinal].} \label{eq:cctoc}
\end{equation}
Here, the decision variable $\tfinal$ represents the time when the terminal state is reached. Furthermore, $\ell(\bm{x}, \bm{u})$ forms the cost function that is used to define the objective functional $J_{\mathrm{tot}}$.

\begin{myassum} \label{ass:ell}
	$\ell: \mathbf{R}^n \times \mathbf{R}^m \rightarrow \mathbf{R}^1$ is continuous, and $\ell(\bm{x},\bm{u}) \geq \varepsilon > 0$ for all $(\bm{x}, \bm{u}) \in \mathcalOld{X} \times  \mathcalOld{U}$.
\end{myassum}

\begin{myrem}
	Assumption~\ref{ass:ell} provides an explicit penalty on the terminal time. Hence,  $J_{\mathrm{tot}} \rightarrow \infty$ as $\tfinal \rightarrow \infty$.
\end{myrem}
One commonly used cost function for motion planning and optimal control problems can be written in the form:
\begin{equation} \label{eq:obj_ref}
\ell(\bm{x}, \bm{u}) = 1 + ||\bm{x} ||_Q^2 + ||\bm{u} ||_R^2,  
\end{equation}
in which the weight matrices $Q \succeq 0$ and $R \succeq 0$ are used to determine the trade-off between time duration (captured by the first term in \eqref{eq:obj_ref}) and other measures such as smoothness of a motion~\citep{ljungqvist2019path}.

As discussed in Section~\ref{sec:intro}, the problem in~\eqref{eq:cctoc} is hard to solve by applying direct optimal control techniques due to the non-convex obstacle avoidance constraints and the nonlinear dynamics. Hence, a good initialization strategy is required to enable the possibility of computing efficient and reliable solutions~\citep{bergman2019bimproved}. In this work, it is assumed that a motion planning algorithm (such as the ones described in Section~\ref{sec:intro}) has provided a nominal trajectory that moves the system from $\xinitial$ to $\xfinal$ and is at least a feasible solution to \eqref{eq:cctoc}. This trajectory is represented by $(\bar{\bm{x}}(\tau), \bar{\bm{u}}(\tau)), \; \tau \in [\tinitial, \; \tfinalnom]$, where $\bar{\bm{x}}(\tau)$ satisfies:
\begin{equation} \label{eq:lat_sol}
\bar{\bm{x}}(\tau) = \xinitial + \int_{\tinitial}^{\tau} f(\bar{\bm{x}}(t), \bar{\bm{u}}(t)) \mathrm{d}t
\end{equation}
This nominal trajectory $(\bar{\bm{x}}(\cdot), \bar{\bm{u}}(\cdot), \tfinalnom)$ is used computationally to warm-start the second RHP step, but also theoretically to guarantee convergence to the terminal state. A detailed description of this procedure is given in the next section.

\section{Receding horizon planning} \label{sec:rhi}
In this section, it will be shown how to use an optimization-based receding horizon planner to optimize a nominal trajectory already computed by a motion planning algorithm. The nominal trajectory is used in the RHP approach to represent a terminal manifold, which ensures the existence of a feasible trajectory to the terminal state beyond the current receding planning horizon. 

\subsection{Receding horizon planning formulation}
The problem of optimizing the nominal trajectory is solved using an iterative receding horizon approach. At each RHP iteration $k$ at time $t_k = t_0 + k\delta, \; {\delta > 0}, \; {k \in \mathbf{Z}_0}$, an OCP is solved over a sliding time window $[\tcurr, \tcurr + T]$, where $T \in (\delta, T_{\text{max}}]$ denotes its length in time. This optimization-based RHP problem is defined as:
\begin{equation}
\minimize{ &J(\bm{x}_{\text{cur}}, \bm{u}_k(\cdot), \tau_k) = \\ &\Psi_k(\tau_k) + \int_{\tcurr}^{\tcurr+T} \hspace{-1.5em} \ell (\bm{x}_k(t), \bm{u}_k(t)) \mathrm{d}t}{\bm{u}_k(\cdot), \;\tau_k}{&\bm{x}_k(\tcurr) = \xcurr,
	\; \; \bm{x}_k(\tcurr + T) = \bar{\bm{x}}_{k-1}(\tau_k)  \\ &\dot{\bm{x}}_k (t) = f(\bm{x}_k(t),\bm{u}_k(t)),  \\ &\bm{x}_k(t) \in \mathcal{X}_{\mathrm{free}}, &&\hspace{-17.8ex}t \in [\tcurr,\tcurr+T] \\
	& \bm{u}_k(t) \in \mathcal{U}. } \label{eq:mpc}
\end{equation}
Here, $\xcurr = \bar{\bm{x}}_{k-1}(t_k)$ is the predicted state of the system at time $\tcurr$, $\bar{\bm{x}}_{k-1}(\cdot)$ the previously optimized state trajectory at time $t_k$ (with $\bar{\bm{x}}_{-1}(\cdot) = \bar{\bm{x}}(\cdot))$ and $\Psi_k(\tau_k)$ the cost-to-go function. Compared to \eqref{eq:cctoc}, a subindex $k$ has been added to the state and control signal to clarify that it is related to the $k$:th RHP iteration. Furthermore, an additional decision variable $\tau_k$ has been added. This variable can be seen as a timing parameter and is used in the terminal constraint to select at what time instance the state at the end of the horizon \mbox{$\bm{x}_k(\tcurr + T)$} is connected to the previously optimized state trajectory $\bar{\bm{x}}_{k-1}(\cdot)$, which defines the terminal manifold. From this state on the terminal manifold, an open-loop control law is known that moves the system from $\bar{\bm{x}}_{k-1}(\tau_k), \tau_k \in [\tinitial, \tfinalnom^{k-1}]$ to $\xfinal$. Note that if the previous solution is already locally optimal, the optimal solution to $\eqref{eq:mpc}$ is given by ($\ustar_k(\cdot), \taustar_k$), where $\ustar_k(t) = \bar{\bm{u}}_{k-1}(t), \; t \in [t_k, t_k + T]$ and $\taustar_k = t_k + T$. Otherwise, a time shift to connect to the previous solution might occur, which is defined as
\begin{equation}
\Delta t_k = \taustar_k - (t_k + T).
\end{equation}
Hence, a new optimized solution $\bar{\bm{u}}_k(\cdot)$ is available in the end of each RHP iteration and is given by 
\begin{equation} \label{eq:tot_control}
\bar{\bm{u}}_k(t) = \begin{cases}
\bar{\bm{u}}_{k-1}(t), &t \in [t_0, t_k) \\
\ustar_k(t) \in \mathcal{U},  &t \in [t_k, t_k + T) \\
\bar{\bm{u}}_{k-1}(t + \Delta t_k) , &t \in [t_k + T, \tfinalnom^{k-1} - \Delta t_k],
\end{cases}
\end{equation} 
where $\bar{\bm{u}}_{-1}(\cdot) = \bar{\bm{u}}(\cdot)$ which is the nominal control trajectory. Furthermore, the new terminal time is updated according to $\tfinalnom^k = \tfinalnom^{k-1} - \Delta t_k$ and the new optimized state trajectory $\bar{\bm{x}}_k(\cdot)$ is defined analogously as in \eqref{eq:tot_control}.

In order to be able to select the optimal choice of $\tau_k$, i.e., where to connect onto the terminal manifold given by $\bar{\bm{x}}_{k-1}(\cdot)$, a terminal cost $\Psi_k(\tau_k)$ is added that represents the cost to transfer the system from $\bar{\bm{x}}_{k-1}(\tau_k)$ to $\xfinal$ using the previously optimized solution. This cost-to-go function is given by
\begin{equation} \label{eq:term_cost}
\Psi_k (\tau_k) = \int_{\tau_k}^{\tfinalnom^{k-1}} \hspace{-1em} \ell (\bar{\bm{x}}_{k-1	}(t), \bar{\bm{u}}_{k-1}(t)) \mathrm{d}t,  \tau_k \in [\tinitial, \tfinalnom^{k-1}],
\end{equation} 
which represents an admissible overestimate of the optimal cost-to-go, obtained from the previous solution. 

\subsection{Feasibility, optimality and convergence} \label{sec:theory}
It will now be shown that the RHP problem in \eqref{eq:mpc} possesses the following properties: i) recursive feasibility, ii) the total objective function value will be non-increasing at every RHP iteration, and iii) convergence to the terminal state. The reasoning behind most of the results are inspired by stability analysis for nonlinear model predictive control (MPC)~\citep{mayne2000constrained}.

\begin{figure}
	\centering
	\setlength\figureheight{0.1667\textwidth}
	\setlength\figurewidth{0.3\textwidth}
	\hspace{2em}\definecolor{mycolor1}{rgb}{0.6510,0.8078,0.8902}%
\definecolor{mycolor2}{rgb}{0.1216,0.4706,0.7059}%
\definecolor{mycolor3}{rgb}{0.6980,0.8745,0.5412}%
\definecolor{mycolor4}{rgb}{0.2000,0.6275,0.1725}%
\definecolor{mycolor5}{rgb}{0.9843,0.6039,0.6000}%
\definecolor{mycolor6}{rgb}{0.8902,0.1020,0.1098}%
\definecolor{mycolor7}{rgb}{ 0.9922,0.7490,0.4353}%
\definecolor{mycolor8}{rgb}{1.0000,0.4980,0}%
\definecolor{mycolor9}{rgb}{ 0.7922,0.6980,0.8392}%
\definecolor{mycolor10}{rgb}{0.4157,0.2392,0.6039}%
\begin{tikzpicture}

\begin{axis}[%
hide axis,
width=0,
height=0,
at={(0,0)},
scale only axis,
xmin=0,
xmax=1,
ymin=0,
ymax=1,
axis background/.style={fill=white},
xmajorgrids,
ymajorgrids,
legend style={legend cell align=left,column sep=3pt, align=center},
legend columns=-1,
]
\addplot [color=mycolor2, line width=1.2pt]
table[row sep=crcr]{%
	1	0.0\\
};
\addlegendentry{ {\small $\bar{\bm{x}}_{k-1}(\cdot)$  } }

\addplot [color=mycolor4, line width=1.2pt]
table[row sep=crcr]{%
	1	0\\
};
\addlegendentry{ {\small $\bm{x}^{\star}_k(\cdot)$ } }

\addplot [color=mycolor8, line width=1.2pt, dashed]
table[row sep=crcr]{%
	1	0\\
};
\addlegendentry{ {\small $ \bar{\bm{x}}_{k-1}\left(\taustar_k : \tfinalnom^{k-1} \right) $} }

\end{axis}
\end{tikzpicture}%
	\input{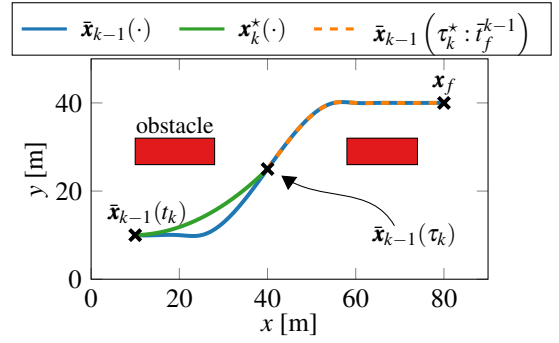} 
	\vspace{-2pt}
	\caption{\small An illustrative example of one RHP iteration. The problem in \eqref{eq:mpc} is solved from $\bar{\bm{x}}_{k-1}(t_k)$, which results in an optimal state trajectory (green). The previous solution $\bar{\bm{x}}_{k-1}(\cdot)$ (blue) is used to provide guarantees that a feasible trajectory to the terminal state exist beyond the receding planning horizon (dashed).   }
	\vspace{-3pt} 
\end{figure}

\begin{mylem}[\textbf{Recursive feasibility}] \label{lem:rec}
	\hfill \\Assume that the nominal trajectory $(\bar{\bm{x}}_{-1}(\cdot), \bar{\bm{u}}_{-1}(\cdot) )$ is feasible in~\eqref{eq:cctoc}. Then, at all RHP iterations $k$ satisfying $t_k + T \leq \tfinalnom^{k-1}$, there exists a feasible solution to \eqref{eq:cctoc}.
\end{mylem}

\begin{proof}
	Assume that $\bar{\bm{u}}_{k-1}(\cdot)$ is feasible in \eqref{eq:cctoc} at RHP iteration $k-1$. Then, at any RHP iteration $k$, \newline \mbox{$\forall k : t_k + T \leq \tfinalnom^{k-1}$}, one choice of feasible decision variables in \eqref{eq:mpc} is:
	\begin{equation} \label{eq:feas_init}
	\begin{aligned}
	\tau^i_k &= t_k + T, \\
	\bm{u}^i_k(t) &= \bar{\bm{u}}_{k-1}(t), \; t \in [t_k, t_k + T).
	\end{aligned}
	\end{equation}
	After solving~\eqref{eq:mpc}, an updated full horizon open-loop control law feasible in \eqref{eq:cctoc} at RHP iteration $k$ is obtained from~\eqref{eq:tot_control} as $\bar{\bm{u}}_{k}(\cdot)$. The desired result follows from induction by noting that at RHP iteration 0, $\bar{\bm{u}}_{-1}(\cdot)$ is feasible.
\end{proof}

\begin{mythm}[\textbf{Full horizon objective function value}] \label{thm:noninc}
	Assume that the nominal trajectory $(\bar{\bm{x}}_{-1}(\cdot), \bar{\bm{u}}_{-1}(\cdot) )$ is feasible in~\eqref{eq:cctoc}. Then, the result in the end of each RHP iteration $k$ satisfying $t_k + T \leq \tfinalnom^{k-1}$ is a full horizon open-loop control law $\bar{\bm{u}}_k(\cdot)$ that is feasible in \eqref{eq:cctoc} and satisfies 
	\begin{equation*}
	J_{\text{tot}}(\xinitial,\bar{\bm{u}}_k(\cdot)) \leq J_{\mathrm{tot}}(\xinitial,\bar{\bm{u}}_{k-1}(\cdot)) \leq \hspace{-2pt}\ldots\hspace{-2pt} \leq J_{\text{tot}}(\xinitial, \bar{\bm{u}}_{-1}(\cdot)).
	\end{equation*}
\end{mythm}

\begin{proof}
	From Lemma~\ref{lem:rec}, it is known that $\bar{\bm{u}}_{k-1}(\cdot)$ is feasible in $\eqref{eq:cctoc}$. Furthermore, the objective function value is $J_{\text{tot}}(\xinitial,\bar{\bm{u}}_{k-1}(\cdot))$, which can be equivalently expanded as
	\begin{equation} \label{eq:obj_val}
	\begin{aligned}
	J_{\text{tot}}(\xinitial, \bar{\bm{u}}_{k-1}(\cdot)) &= \Psi_{ctc}(t_{k-1}) \\ &+ J(\bar{\bm{x}}_{k-1}(t_{k-1}),\bm{u}^{\star}_{k-1}(\cdot), \taustar_{k-1}),
	\end{aligned}
	\end{equation}
	where $\Psi_{ctc}(t) $ is the cost-to-come function, i.e., the accumulated cost up until $t$, with $\Psi_{ctc}(\tinitial) = 0$, while $J$ and ($\ustar_{k-1}(\cdot), \taustar_{k-1}$) are the objective function and the solution to $\eqref{eq:mpc}$ at RHP iteration $k-1$, respectively. By using \eqref{eq:mpc}, \eqref{eq:tot_control}, \eqref{eq:term_cost}, \eqref{eq:feas_init} in \eqref{eq:obj_val}, it follows that
	\begin{equation}
	\begin{aligned}
	&\hspace{-0.5em}J_{\text{tot}}(\xinitial, \bar{\bm{u}}_{k-1}(\cdot)) = \\
	&\hspace{-0.5em}\underbrace{\Psi_{ctc}(t_{k-1}) + \int_{t_{k-1}}^{t_{k}} \hspace{-1em} \ell(\bar{\bm{x}}_{k-1}(t),\bar{\bm{u}}_{k-1}(t)) \mathrm{d}t}_{\Psi_{ctc}(t_k)} \\ 
	&\hspace{-0.5em}+\underbrace{\int_{t_{k}}^{t_{k}+T} \hspace{-1.5em}\ell(\bar{\bm{x}}_{k-1}(t),\bar{\bm{u}}_{k-1}(t)) \mathrm{d}t + \Psi_k(t_k + T)}_{\text{Using \eqref{eq:feas_init} in \eqref{eq:mpc} : }J(\bar{\bm{x}}_k(t_k),\bm{u}^i_k(\cdot), \tau^i_k)} = \\ 
	&\hspace{-0.5em}\Psi_{ctc}(t_k) + J(\bar{\bm{x}}_k(t_k),\bm{u}^i_k(\cdot), \tau^i_k) \geq \\
	&\hspace{-0.5em}\Psi_{ctc}(t_k) + J(\bar{\bm{x}}_k(t_k),\bm{u}^{\star}_k(\cdot), \taustar_k) = J_{\text{tot}}(\xinitial, \bar{\bm{u}}_k(\cdot)). \\
	\end{aligned}
	\end{equation}
	Thus, using induction, it is possible to conclude that:
	\begin{equation*}
	\begin{aligned}
	J_{\text{tot}}(\xinitial,\bar{\bm{u}}_k(\cdot)) \leq J_{\mathrm{tot}}(\xinitial,\bar{\bm{u}}_{k-1}(\cdot)) \leq \hspace{-2pt}\ldots\hspace{-2pt} \leq J_{\text{tot}}(\xinitial, \bar{\bm{u}}_{-1}(\cdot)).
	\end{aligned}
	\end{equation*}
	which holds $\forall k : t_k + T \leq \tfinalnom^{k-1}$. When $t_k + T > \tfinalnom^{k-1}$, an optimal solution within the current planning horizon already exists and no re-planning is required.	
\end{proof}

\begin{myrem}
	Note that Assumption~\ref{ass:ell} on the cost function $\ell(\bm{x},\bm{u})$ is not required in Lemma~\ref{lem:rec} nor in Theorem~\ref{thm:noninc}.	
\end{myrem}

\begin{myrem} \label{rem:poo}
	When $t_k + T > \tfinalnom^{k-1}$, one possibility is to perform re-planning by iteratively decreasing the planning horizon $T$. However, the optimal solution will stay the same during these last $T/\delta$ RHP iterations using arguments from principle of optimality.
\end{myrem}

\begin{mythm}[\textbf{Finite number of RHP iterations}] \label{thm:finite}
	\hfill \\ Under Assumption~\ref{ass:ell}, the maximum number of RHP iterations $k_{\text{max}}$ is upper bounded by 
	\begin{equation} \label{eq:kmax}
	k_{\text{max}} \leq \frac{J_{\text{tot}}(\xinitial, \bar{\bm{u}}_{-1}(\cdot))}{\varepsilon \delta},
	\end{equation}
	where $\delta$ is the time between two consecutive RHP iterations.
\end{mythm}

\begin{proof}
	At RHP iteration $k$, Assumption~\ref{ass:ell} and \eqref{eq:obj_val} give
	\begin{equation} \label{eq:lower_bound}
	\begin{aligned}
	&J_{\text{tot}}(\xinitial, \bar{\bm{u}}_k(\cdot)) \geq \Psi_{ctc}(t_k) = \\
	&\int_{\tinitial}^{\tinitial+\delta k} \hspace{-1em} \underbrace{\ell(\bar{\bm{x}}_k(t), \bar{\bm{u}}_k(t) )}_{\geq \varepsilon} \mathrm{d}t \geq  \varepsilon \delta k.
	\end{aligned}		 
	\end{equation}
	From Theorem~\ref{thm:noninc}, it holds that 
	\begin{equation*}
	J_{\text{tot}}(\xinitial, \bar{\bm{u}}_k(\cdot)) \leq J_{\text{tot}}(\xinitial, \bar{\bm{u}}_{-1}(\cdot) ), \forall k : \; t_k + T \leq \tfinalnom^{k-1}
	\end{equation*}
	which combined with \eqref{eq:lower_bound} gives
	\begin{equation}
	\varepsilon \delta k \leq J_{\text{tot}}(\xinitial, \bar{\bm{u}}_{-1}(\cdot)) \iff k \leq \frac{J_{\text{tot}}(\xinitial, \bar{\bm{u}}_{-1}(\cdot))}{\varepsilon\delta },
	\end{equation}
	which completes the proof. 
\end{proof}

\begin{mycor}[\textbf{Convergence to terminal state}]
	\hfill \\ Under Assumption~\ref{ass:ell}, the terminal state $\xfinal$ will be reached in finite time. 
\end{mycor}

\begin{proof}
	Using Theorem~\ref{thm:finite}, the terminal time $\tfinal$ when the terminal state $\xfinal$ is reached is upper bounded by
	\begin{equation}
	\tfinal \leq \tinitial + \delta k_{\text{max}} + T,
	\end{equation}
	where $k_{\text{max}}$ is upper bounded in~\eqref{eq:kmax} and $T$ is the user-defined RHP horizon length in~\eqref{eq:mpc}. 
\end{proof}

\section{A practical algorithm} \label{sec:alg}

In this section, a reformulation of the RHP problem in the previous section is introduced to handle a piecewise continuous nominal control trajectory. The new formulation is connected to the theory in Section~\ref{sec:rhi} to show that recursive feasibility, non-increasing objective function value and convergence to the terminal state still can be guaranteed. Finally, an algorithm is outlined which summarizes all steps in the proposed RHP approach.

\subsection{Solving the receding horizon planning problem} \label{sec:problems}
A common approach to solve OCPs such as the RHP problem in \eqref{eq:mpc} is to use direct methods for optimal control. In these methods, the continuous problem is discretized and cast as a standard NLP. This is typically achieved by using a piecewise continuous control signal~\citep{diehl2006fast}. The discretized problem can then be solved using standard methods for nonlinear optimization such as SQP or nonlinear interior point methods~\citep{nocedal2006numerical}. These solvers can be interfaced through a standard solver interface such as CasADi~\citep{andersson2018casadi}, which can be used when all involved functions in~\eqref{eq:mpc} are (at least) continuously differentiable everywhere. 

In practice, it is desirable to use nominal trajectories in \eqref{eq:mpc} where the control signal is piecewise continuous. As an example, this is the case when a lattice-based motion planner is used to compute a nominal trajectory using motion primitives computed by applying direct optimal control techniques~\citep{bergman2019improved}. The problem of using a piecewise continuous nominal control signal trajectory is that the terminal manifold, defined by $\bar{\bm{x}}_{k-1}(\tau)$, and the cost-to-go function $\Psi_k(\tau)$ in \eqref{eq:mpc} are piecewise continuously differentiable with respect to the timing variable $\tau$. This follows from that
\begin{equation}
\begin{aligned}
\frac{\mathrm{d}\bar{\bm{x}}_{k-1}}{\mathrm{d}\tau} &= \dot{\bar{\bm{x}}}_{k-1}(\tau) = f(\bar{\bm{x}}_{k-1}(\tau), \bar{\bm{u}}_{k-1}(\tau)),  \\
\frac{\mathrm{d}\Psi_k}{\mathrm{d}\tau} &= - \ell(\bar{\bm{x}}_{k-1}(\tau) , \bar{\bm{u}}_{k-1}(\tau) ),
\end{aligned}
\end{equation}
explicitly depend on the piecewise continuous control signal trajectory $\bar{\bm{u}}_{k-1}(\tau)$. Hence, in this case it is not possible to directly use standard solver interfaces. One possibility is to modify the solver and/or solver interface, which is out of scope in this work. Another possibility, which is used in this paper and will further be described in the next sections, is to adjust the problem formulation while aiming at preserving the theoretical guarantees proved in Section~\ref{sec:theory}.

\subsection{Adjusted receding horizon planning formulation} \label{sec:rvmpc}
One approach to deal with a piecewise continuous nominal control trajectory is to use a variable horizon length $T_k$ in each RHP iteration, and select the value of the timing parameter $\tau_k$ in \eqref{eq:mpc} in a separate step. This means that the RHP problem in \eqref{eq:mpc} can be reformulated as:
\begin{equation}
\minimize{ J = \int_{\tcurr}^{\tcurr+T_{k}} \ell (\bm{x}_k(t), \bm{u}_k(t)) \mathrm{d}t}{\bm{u}_k(\cdot), \;T_k}{&\bm{x}_k(\tcurr) = \xcurr, \\ &
	\bm{x}_k(\tcurr + T_k) = \bar{\bm{x}}_{k-1}(\tau_k)  \\ &\dot{\bm{x}}_k (t) = f(\bm{x}_k(t),\bm{u}_k(t)),  \\ &\bm{x}_k(t) \in \mathcal{X}_{\mathrm{free}}, 
	&& \hspace{-10ex} \bm{u}_k(t) \in \mathcal{U}. } \label{eq:mpc_variable_T}
\end{equation}
Here, the difference compared to \eqref{eq:mpc} is that $T_k$ is added as a decision variable, and $\tau _k$ is removed from being a decision variable and is instead considered as a parameter to the RHP problem. Since $\tau_k$ is no longer a decision variable, it is not an issue with using piecewise continuously differentiable functions $\bar{\bm{x}}_{k-1}(\cdot)$ and $\Psi_k(\cdot)$. This new problem formulation reduces the terminal state manifold to a single state. Furthermore, the cost-to-go function $\Psi_k(\cdot)$ does not need to be explicitly taken into account since the terminal state, and hence also the cost along the remaining nominal solution, is already selected before \eqref{eq:mpc_variable_T} is solved. By assuming a piecewise continuous input over each planning interval $[t_k, t_{k+1}]$, the problem can thus be discretized using direct optimal control methods and solved using standard NLP interfaces.  

\subsection{Feasibility, optimality and convergence}

The theoretical results in Section~\ref{sec:theory} neglected that the RHP problem is to be discretized when solved using direct optimal control techniques. This discretization introduces the possibility of loosing recursive feasibility (in contrast to the theoretical setup in Lemma~\ref{lem:rec}) since it is not guaranteed that the time-shifted input in~\eqref{eq:feas_init} is possible to represent in the discretized version. Even if the problem turns out to be feasible, it could be the case that Theorem~\ref{thm:noninc} does not hold, i.e., the new solution has a higher objective function value than the previously optimized solution. Here, we show how to obtain a practical implementation with the properties already guaranteed for the somewhat simplified theoretical setup in Section~\ref{sec:rhi}.

At RHP iteration $k-1$, $(\bar{\bm{x}}_{k-1}(t),\bar{\bm{u}}_{k-1}(t))$ is executed during the time interval $t \in [t_{k-1}, t_{k}]$. Since both model errors and external disturbances are assumed to be zero, the state at $t_{k}$ will be $\bar{\bm{x}}_{k-1}(t_{k})$. By setting $\xcurr = \bar{\bm{x}}_{k-1}(t_{k})$ and a desired value of $\tau_{k}$ in \eqref{eq:mpc_variable_T}, the solution at RHP iteration $k$ (if any exists) will be given by ($\ustar_{k}(\cdot), \tstar_{k})$. If the problem is feasible, a new candidate nominal control is to use:
\begin{equation} \label{eq:new_cand}
\bar{\bm{u}}_{\text{can}}(t) = \begin{cases}\bar{\bm{u}}_{k-1}(t), & \hspace{-0.7em}t \in [t_0, t_k)  \\\ustar_{k}(t), &\hspace{-0.7em}t \in [t_{k}, t_{k}+\tstar_k) \\ \bar{\bm{u}}_{k-1}(t+\Delta t_{k}), & \hspace{-0.7em}t \in [t_{k} + \tstar_k, \tfinalnom^{\text{can}} ] \end{cases}
\end{equation}
where $\Delta t_{k} = \tau_{k} - (t_{k} + \tstar_{k})$. In order to guarantee a result similar to Theorem~\ref{thm:noninc}, the candidate solution is explicitly benchmarked against the old one $\bar{\bm{u}}_{k-1}(\cdot)$. If the total objective function value is improved by using the new candidate, i.e.
\begin{equation} \label{eq:cand_comp}
J_{\text{tot}}(\xinitial,\bar{\bm{u}}_{\text{can}}(\cdot)) < J_{\text{tot}}(\xinitial,\bar{\bm{u}}_{k-1}(\cdot)) 
\end{equation}
the nominal trajectory is updated:
\begin{equation}
\left(\bar{\bm{u}}_k(\cdot), \bar{\bm{x}}_k(\cdot), \tfinalnom^k \right)   = \left( \bar{\bm{u}}_{\text{can}}(\cdot), \; \bar{\bm{x}}_{\text{can}}(\cdot), \; \tfinalnom^{k-1}  - \Delta t_{k} \right),
\end{equation}
where $\bar{\bm{x}}_{\text{can}}(\cdot)$ can be computed analogously to $\bar{\bm{u}}_{\text{can}}(\cdot)$ in \eqref{eq:new_cand}. Otherwise, the previously optimized solution $\bar{\bm{u}}_{k-1}(\cdot)$ is reused, which still represents a feasible solution to $\xfinal$.  Hence, a practically useful approach that provides similar guarantees as in Lemma~\ref{lem:rec} and Theorem~\ref{thm:noninc} is obtained using~\eqref{eq:new_cand} and \eqref{eq:cand_comp}. Another required property is to ensure that the approach converges to the terminal state $\xfinal$. Since the timing variable $\tau _k$ is updated before and kept fixed during each RHP iteration (as described in Section~\ref{sec:rvmpc}), progress towards $\tfinalnom^{k-1}$ is required for convergence. A sufficient condition for progress is
\begin{equation} \label{eq:update_policy}
\tau_{k+1} \geq \tau_k + \varepsilon_\tau,
\end{equation} 
which means that $\tau_k = \tfinalnom^{k-1} < \infty$ will be selected after a finite number of RHP iterations, implying that $\xfinal$ is used as terminal state in \eqref{eq:mpc_variable_T} and hence eventually reached.  
\subsection{Algorithm}

\begin{algorithm}[t]
	\caption{Receding horizon planning }
	\label{alg:rhi}
	\begin{algorithmic}[1]
		\State \textbf{Input}: $\xinitial, \xfinal$, $T$, $\delta$, $\mathcal{X}_{\text{free}}$
		\State ($\bar{\bm{x}}_{-1}, \bar{\bm{u}}_{-1}, \tfinalnom) \leftarrow$ Motion planner($\xinitial, \xfinal, \mathcal{X}_{\text{free}}$)
		\State $\tau_0 \leftarrow t_0 + T$, \quad $T_0^{\text{init} } \leftarrow \tau_0 - t_0$ 
		\State ($\bm{x}^{\text{init}}_0, \bm{u}^{\text{init}}_0)  \leftarrow $ resample($\bar{\bm{u}}_{-1}, \bar{\bm{x}}_{-1}, \delta $ )
		\While{$\tau_k \neq \tau_ {k-1}$ }
		\State Set $\xcurr = \bar{\bm{x}}_{k-1}(t_k)$ in \eqref{eq:mpc_variable_T} 
		\State ($\ustar_k, \tstar_k) \leftarrow $ Solve \eqref{eq:mpc_variable_T} using $\bm{u}^{\text{init}}_k, \bm{x}^{\text{init}}_k , T_k^{\text{init}}$ and $\tau_k$
		\If{$J(\xcurr, \ustar_k, \tstar_k) < \infty$}
		\State $\Delta t_k \leftarrow \tau_k - (t_k + \tstar_k)$
		\State ($\bar{\bm{u}}_{\text{can}}, \bar{\bm{x}}_{\text{can}}) \leftarrow$ get\_cand($\bar{\bm{u}}_{k-1}, \bar{\bm{x}}_{k-1}, \ustar_k, \Delta t_k $)
		\If{$J_{\text{tot}}(\xinitial,\bar{\bm{u}}_{\text{can}}) < J_{\text{tot}}(\xinitial,\bar{\bm{u}}) $}
		\State Update solution:
		
		\hspace{5ex} ($\bar{\bm{u}}_k, \bar{\bm{x}}_k) \leftarrow (\bar{\bm{u}}_{\text{can}}, \bar{\bm{x}}_{\text{can}})$
		
		\hspace{6ex}$\tfinalnom^k \leftarrow \tfinalnom^{k-1} - \Delta t_k$
		\Else{}
		\State ($\bar{\bm{u}}_k, \bar{\bm{x}}_k,\tfinalnom^k) \leftarrow (\bar{\bm{u}}_{k-1}, \bar{\bm{x}}_{k-1}, \tfinalnom^{k-1})$ 
		\EndIf
		\Else{}
		\State ($\bar{\bm{u}}_k, \bar{\bm{x}}_k, \tfinalnom^k) \leftarrow (\bar{\bm{u}}_{k-1}, \bar{\bm{x}}_{k-1}, \tfinalnom^{k-1})$
		\EndIf
		\State Send nominal trajectory to controller :
		
		send\_reference($\bar{\bm{u}}_k, \bar{\bm{x}}_k$) 
		\State Update receding horizon terminal constraint:
		
		$\tau_{k+1} \leftarrow $ update\_timing($t_{k+1}, T, \tfinalnom^k$)
		\State Initialization for next iteration:
		
		$T_{k+1}^{\text{init}} \leftarrow \tau_{k+1} - t_{k+1}$ 
		
		$\bm{x}^{\text{init}}_{k+1}, \bm{u}^{\text{init}}_{k+1} \leftarrow $ resample($\bar{\bm{u}}_k, \bar{\bm{x}}_k, T^{\text{init}}_{k+1}/N $ )
		\State Set $ k \rightarrow k+1$
		\EndWhile
		
	\end{algorithmic}
\end{algorithm}
The resulting RHP algorithm for motion planning is outlined in Algorithm~\ref{alg:rhi}. Before explaining the steps, note that state and control signal trajectories, i.e. $\bm{x}(\cdot)$ and $\bm{u}(\cdot)$ in Algorithm~\ref{alg:rhi}, are written as $\bm{x}$ and $\bm{u}$ for notational brevity. 

The inputs to the algorithm are given by the initial and terminal states, a desired planning horizon $T$, the time between two consecutive RHP iterations  $\delta$ (which together define the number of discretization points $N = T/\delta)$, and the current representation of $\mathcal{X}_{\text{free}}$. A motion planner is then used on Line 2 to compute a nominal trajectory. To obtain the best overall performance, the nominal trajectory should also be computed while minimizing the same objective function value as in~\eqref{eq:cctoc}~\citep{bergman2019bimproved}, since the RHP iterations only perform local improvements of the nominal trajectory. 

For each RHP iteration $k$, the problem in \eqref{eq:mpc_variable_T} is solved from $\xcurr = \bar{\bm{x}}_{k-1}(t_k)$ starting from a provided initialization (discussed further down in this section) and a selected value of $\tau_k$. If this problem is feasible, a new candidate solution is found using \eqref{eq:new_cand}. If this candidate has a lower full horizon objective function value (i.e. the inequality in \eqref{eq:cand_comp} holds), the current candidate is selected as solution. Otherwise, the previous solution is reused. The selected solution is sent on Line 19 to a trajectory-tracking controller.

The timing variable $\tau_k$ is updated at Line 20 in Algorithm~\ref{alg:rhi}. The result in \eqref{eq:update_policy} only requires an update policy such that $\tau_{k+1} \geq \tau_k + \varepsilon_\tau$. One policy that satisfies this requirement is:
\begin{equation}
\tau_{k+1} = \min \left(\tfinalnom^k, t_{k+1} + T \right),
\end{equation}
since $t_{k+1}+T = \tau_k + \delta$. This means that the terminal state at the next RHP iteration is selected using the user-defined desired planning horizon $T$ in Algorithm~\ref{alg:rhi}.

Finally, the solver initialization for the next RHP iteration is done on Line 21 in Algorithm~\ref{alg:rhi}. First, $T_k$ is initialized according to the predicted length, i.e., $T^{\text{init}}_{k+1} = \tau_{k+1} - t_{k+1}$. Then, the previous full horizon solution is resampled to be compatible with $T^{\text{init}}_{k+1}$. Assuming a piecewise constant control signal and a multiple-shooting discretization strategy, one possible resampling of $(\bar{\bm{x}}_k(\cdot), \bar{\bm{u}}_k(\cdot) )$ is
\begin{equation} \label{eq:init}
\begin{aligned}
\bm{u}_{k+1}^{\text{init}}(t_j) &= \bar{\bm{u}}_k(t_j), \quad \forall j \in [k+1, k+1+N], \\
\bm{x}_{k+1}^{\text{init}}(t_j) &= \bar{\bm{x}}_k(t_j), \quad \forall j \in [k+1,k+2+N],
\end{aligned}
\end{equation}
where $N$ represents the number of discretization points (given by $T/\delta$), and $t_j = \tinitial + j \delta^{\text{init}}$, with $\delta^{\text{init}} = T^{\text{init}}_{k+1}/N$.
The RHP iterations are solved until $\tau_{k} = \tau_{k-1}$, which means that $\xfinal$ has been used as terminal state in~\eqref{eq:mpc_variable_T}.

\section{Simulation study} \label{sec:Res}
In this section, the proposed optimization-based RHP approach presented in Section~\ref{sec:alg} is evaluated in two challenging parking problem scenarios for a truck and trailer system. To evaluate the proposed RHP approach, a lattice-based motion planning algorithm is employed in a first step to compute nominal trajectories using a library of precomputed motion primitives. The lattice-based planner is implemented in C++, while the optimization-based RHP approach is implemented in Python using CasADi together with the warm-start friendly SQP solver WORHP~\citep{bueskens2013worhp}.
\vspace*{-0.5em}
\subsection{Vehicle model}
\vspace*{-0.5em}
The truck and trailer system is a general 2-trailer with car-like truck~\citep{altafini2002hybrid,ljungqvist2019path}. The system consists of three vehicle segments: a car-like truck, a dolly and a semitrailer. The state vector for the system is given by 
\begin{equation} \label{eq:states}
\begin{aligned}
\bm{x} &= 
\begin{bmatrix}
\bm{q}^T & \alpha & \omega & v_1&  a_1
\end{bmatrix}^T \\
\bm{q} &= \begin{bmatrix}
x_3 & y_3 & \theta_3 & \beta_3 & \beta_2
\end{bmatrix}^T 
\end{aligned}	
\end{equation}
where $(x_3, y_3)$ and $\theta _3$ represent the position and orientation of the semitrailer, respectively, while $\beta _3$ and $\beta_2$ denote the joint angles between the semitrailer and the truck. Finally, $\alpha$ and $\omega$ are the truck's steering angle and steering angle rate, respectively, while $v_1$ and $a_1$ are the longitudinal velocity and acceleration of the truck. Assuming low-speed maneuvers, the truck and trailer system can compactly be modeled as~\citep{ljungqvist2019path}:
\begin{equation} \label{eq:truckModel}
\begin{aligned}
&\dot{\bm{q}} = v_1 f(\bm{q}, \alpha), \\
& \dot{\alpha} =  \omega, \quad \dot{\omega} = u_{\omega},  \\
& \dot{v}_1  = a_1  \quad  \dot{a}_1 = u_a. \\
\end{aligned}
\end{equation}
\begin{figure*}[t!] 
	\hspace{3.45em} \vspace{-2.8em}\definecolor{mycolor1}{rgb}{0.6510,0.8078,0.8902}%
\definecolor{mycolor2}{rgb}{0.1216,0.4706,0.7059}%
\definecolor{mycolor3}{rgb}{0.6980,0.8745,0.5412}%
\definecolor{mycolor4}{rgb}{0.2000,0.6275,0.1725}%
\definecolor{mycolor5}{rgb}{0.9843,0.6039,0.6000}%
\definecolor{mycolor6}{rgb}{0.8902,0.1020,0.1098}%
\definecolor{mycolor7}{rgb}{ 0.9922,0.7490,0.4353}%
\definecolor{mycolor8}{rgb}{1.0000,0.4980,0}%
\definecolor{mycolor9}{rgb}{ 0.7922,0.6980,0.8392}%
\definecolor{mycolor10}{rgb}{0.4157,0.2392,0.6039}%
\begin{tikzpicture}

\begin{axis}[%
hide axis,
width=0,
height=0,
at={(0,0)},
scale only axis,
xmin=0,
xmax=1,
ymin=0,
ymax=1,
axis background/.style={fill=white},
xmajorgrids,
ymajorgrids,
legend style={legend cell align=left,column sep=3pt, align=center},
legend columns=-1,
]
\addplot [color=mycolor2, dashdotted, line width=1.4pt]
table[row sep=crcr]{%
	1	0.0\\
};
\addlegendentry{ {\small Nominal  } }

\addplot [color=mycolor4, dashed, line width=1.4pt]
table[row sep=crcr]{%
	1	0\\
};
\addlegendentry{ {\small Full Horizon (FH) } }

\addplot [color=mycolor8, line width=1.4pt]
table[row sep=crcr]{%
	1	0\\
};
\addlegendentry{ {\small $ T = 60 $ s} }

\end{axis}
\end{tikzpicture}%
 \\
	\subfloat[][\small{Reverse parking scenario}]{
		\setlength\figureheight{0.1688\textwidth}
		\setlength\figurewidth{0.45\textwidth}
		\input{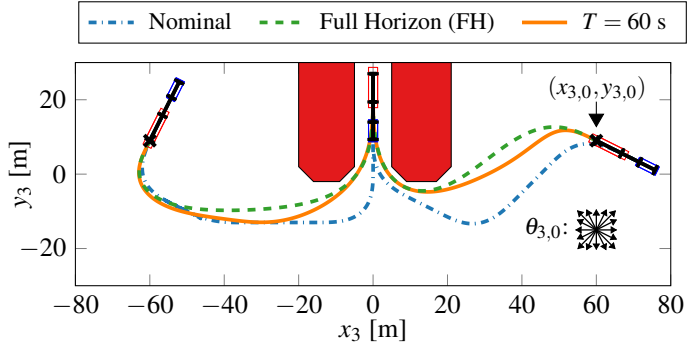}
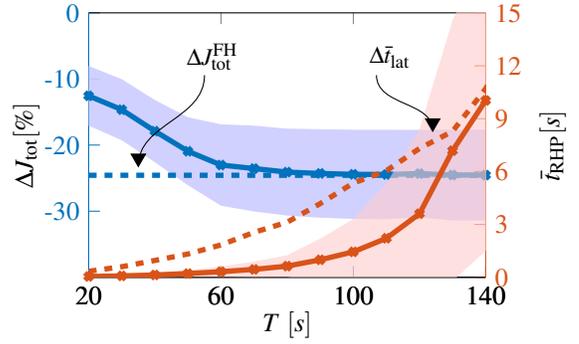 
		\label{fig:rp_ill}
	} \subfloat[][\small Improvement and computation time vs. planning horizon] {
	\setlength\figureheight{0.2\textwidth}
	\setlength\figurewidth{0.3\textwidth}
\definecolor{mycolor1}{rgb}{0.00000,0.44700,0.74100}%
\definecolor{mycolor2}{rgb}{0.85000,0.32500,0.09800}%
\begin{tikzpicture}

\begin{axis}[%
width=\figurewidth,
height=\figureheight,
at={(0\figurewidth,0\figureheight)},
scale only axis,
xmin=20,
xmax=140,
separate axis lines,
every outer y axis line/.append style={mycolor1},
every y tick label/.append style={font=\color{mycolor1}},
every y tick/.append style={mycolor1},
ymin=60,
ymax=100,
ytick={ 70, 80,  90, 100},
yticklabels={ -30, -20,  -10, 0},
%axis y discontinuity=crunch, enlargelimits=false,
%axis x line*=bottom,
axis y line*=left,
xtick={20, 60, 100, 140},
ylabel style={yshift=-0.4cm}, %shifting the y line text
xlabel style={yshift=0.2cm}, %shifting the y line text
ylabel={$\Delta J_{\text{tot}} [\%] $},
xlabel={$T$ $[s]$},
axis background/.style={fill=white}
]
\addplot[area legend, draw=none, fill opacity=0.9, fill=white!80!blue, forget plot]
table[row sep=crcr] {%
x	y\\
20	91.964538995485\\
30	89.9737357729241\\
40	86.8239820608943\\
50	84.2085806831557\\
60	83.1372939423152\\
70	82.9413826589032\\
80	82.4834505549378\\
90	82.3973973919015\\
100	82.3221492242033\\
110	82.3138075543878\\
120	82.3000815789791\\
130	82.3339240686393\\
140	82.3350942925915\\
150	82.32697846651\\
150	68.6221280810266\\
140	68.581438390167\\
130	68.6090659394858\\
120	69.073011184185\\
110	68.7880265490509\\
100	68.8276822301188\\
90	69.0139338310645\\
80	69.3351571865363\\
70	69.9579812134639\\
60	70.8216965968666\\
50	73.9264460899872\\
40	77.3998981666605\\
30	80.734591612665\\
20	82.9355752348994\\
20	91.964538995485\\
}--cycle;
\addplot [color=mycolor1, line width=2.0pt,mark=x, forget plot]
table[row sep=crcr]{%
20	87.4500571151922\\
30	85.3541636927945\\
40	82.1119401137774\\
50	79.0675133865715\\
60	76.9794952695909\\
70	76.4496819361836\\
80	75.9093038707371\\
90	75.705665611483\\
100	75.5749157271611\\
110	75.5509170517193\\
120	75.686546381582\\
130	75.4714950040625\\
140	75.4582663413793\\
150	75.4745532737683\\
};
\addplot [color=mycolor1, dashed, line width = 2.0pt, forget plot]
table[row sep=crcr]{%
	20	75.4317450370988\\
	150	75.4317450370988\\
};
\end{axis}

\begin{axis}[%
width=\figurewidth,
height=\figureheight,
at={(0\figurewidth,0\figureheight)},
scale only axis,
xmin=20,
xmax=140,
xtick style={draw=none},
every outer y axis line/.append style={mycolor2},
every y tick label/.append style={font=\color{mycolor2}},
every y tick/.append style={mycolor2},
ymin=0,
ymax=15,
ytick={0, 3, 6, 9, 12,15},
xtick={20, 60, 100, 140},
ylabel style={yshift=0.9cm}, %shifting the y line text
xlabel style={yshift=0.2cm}, %shifting the y line text
ylabel={$\bar{t}_{\text{RHP}} [s] $},
axis x line*=bottom,
axis y line*=right, ylabel near ticks,
]
\addplot[area legend, draw=none, fill opacity=0.6, fill=white!80!red, forget plot]
table[row sep=crcr] {%
x	y\\
20	0.00443528256005801\\
30	0.00357134259229305\\
40	0.0137362157406317\\
50	0.0244611588884274\\
60	0.0241030939327858\\
70	0.00609566621905361\\
80	0.0151890852552337\\
90	-0.205496681816805\\
100	-0.416851939179794\\
110	-0.893822747901095\\
120	-1.19887533942957\\
130	-0.209409441315991\\
140	1.49865437795503\\
150	4.2425527213251\\
150	24.7355447752856\\
140	18.5680921264554\\
130	14.5859830711843\\
120	8.41515565673608\\
110	5.32919574537506\\
100	3.31128420017491\\
90	2.19808057031062\\
80	1.25887064606572\\
70	0.897176022973909\\
60	0.618837019565481\\
50	0.407018005821784\\
40	0.25911605507018\\
30	0.163064384465326\\
20	0.0970404818644676\\
20	0.00443528256005801\\
}--cycle;
\addplot [color=mycolor2, line width=2.0pt, mark=x, forget plot]
table[row sep=crcr]{%
20	0.0507378822122628\\
30	0.0833178635288094\\
40	0.136426135405406\\
50	0.215739582355106\\
60	0.321470056749133\\
70	0.451635844596481\\
80	0.637029865660478\\
90	0.996291944246906\\
100	1.44721613049756\\
110	2.21768649873698\\
120	3.60814015865326\\
130	7.18828681493417\\
140	10.0333732522052\\
150	14.48904874830540\\
};
\addplot [color=mycolor2, dashed, line width = 2.0pt, forget plot]
table[row sep=crcr]{%
	20	0.344369530677795\\
	30	0.604367926716805\\
	40	0.958723664283752\\
	50	1.32527716457844\\
	60	1.84156015515327\\
	70	2.5543914437294\\
	80	3.12871515750885\\
	90	4.18559192866087\\
	100	5.34836608171463\\
	110	6.05342035740614\\
	120	7.31096279621124\\
	130	8.2842253446579\\
	140	10.7025585398078\\
	150	13.9381398186088\\
};
\node[anchor=east] (t) at (axis cs:68, 12.5){\small{$\Delta J^{\text{FH}}_{\text{tot}}$}};
\node[anchor=east] (t) at (axis cs:120, 12.5){\small{$\Delta\bar{t}_{\text{lat}}$}};
\node (s2) at (axis cs:55,12){};
\node (p2) at (axis cs:35,5.7){};
\draw[->,-triangle 60] (s2) to [out=-90,in=90] (p2);
\node (s3) at (axis cs:107,12){};
\node (p3) at (axis cs:124,7.7){};
\draw[->,-triangle 60] (s3) to [out=-90,in=90] (p3);

\end{axis}
\end{tikzpicture}%
 \label{fig:rp_res}\vspace{-10pt} 
} 
\vspace{-6pt}
\caption{\small{(a): Reverse parking scenario from 32 different initial states. The nominal path (dashdotted) compared to the paths after applying the RHP algorithm using $T=60$ (solid) and the path using full horizon (FH) improvement (dashed). (b): The average difference in objective function value $\Delta J_{\text{tot}}$, and the average computation time per RHP iteration $\bar{t}_{\text{RHP}}$ using different planning horizons $T$ in Algorithm~\ref{alg:rhi}. The shaded area represents $\pm$ one standard deviation. Finally, $\Delta J_{\text{tot}}^{\text{FH}}$ (dashed blue) represents average difference in objective function value using FH improvement, and $\Delta \bar{t}_{\text{lat}}$ (dashed red) is the average difference in latency time.}} \label{fig:rv}
\end{figure*}
The control signal to the truck and trailer system is \mbox{$\bm{u}^T = [u_{\omega}\hspace{5pt} u_{a}]$}. The vehicle's geometry coincides with the one used in~\cite{ljungqvist2019path}. The control signal and the vehicle states are constrained as 
\begin{equation*}
\begin{aligned}
&|\beta_3| \leq 0.87, &&|\beta_2| \leq 0.87, &&& |\alpha| \leq 0.73, &&&& |\omega| \leq  0.8, \\
&|v_1| \leq 1.0, &&|a_1| \leq 1.0, &&& |u_\omega| \leq 10, &&&&|u_a| \leq 40,  
\end{aligned}
\end{equation*}
and the cost function is chosen as
\begin{equation} \label{eq:cost}
\ell(\bm{x}, \bm{u}) = 1 + \frac{1}{2}\left(\alpha^2 + 10\omega^2 + a_1^2 + \bm{u}^T\bm{u}  \right),
\end{equation}
which is used both in the lattice-based planner \emph{and} the proposed optimization-based RHP approach as suggested in~\cite{bergman2019bimproved}.
\vspace*{-0.5em}
\subsection{Lattice-based motion planner}
\vspace*{-0.5em}
As previously mentioned, a lattice-based planner is used in a first step to compute a nominal trajectory to the terminal state. The lattice-based planner uses a discretized state space $\mathcal{X}_d$ and a library of precomputed motion primitives $\mathcal P$. During online planning, a nominal trajectory to the terminal state is computed using A$^{\star}$ graph search together with a precomputed free-space heuristic look-up table (HLUT)~\citep{knepper2006high}. %For this, a state space discretization $\mathcal{X}_d$ need to be selected. 
In this work, we use a similar state-space discretization $\mathcal{X}_d$ as in~\cite{ljungqvist2019path}, where the position of the semitrailer is discretized to a uniform grid with resolution $r=1$ m and the orientation of the semitrailer is irregularly discretized $\theta_3 \in \Theta$ into $|\Theta|=16$ different orientations. It is done to be able to compute short straight trajectories from each $\theta_3 \in \Theta$~\citep{pivtoraiko2009differentially}. One difference compared to~\cite{ljungqvist2019path} is that the longitudinal velocity is here also discretized as \mbox{$v_1 \in \mathcal{V} = \{-1, 0, 1\}$}. All other vehicle states are constrained to zero for all discrete states in $\mathcal{X}_d$ as was done in~\cite{ljungqvist2019path}. Note, however, that on the trajectory between two states in $\mathcal{X}_d$, the system is free to take any feasible state. 

The motion primitive set $\mathcal P$ is computed offline using the framework presented in~\cite{bergman2019improved} and consists of straight, parallel and heading change maneuvers between discrete states in $\mathcal{X}_d$. Velocity changes between discrete states are only allowed during straight motions. At each discrete state with nonzero velocity, heading change maneuvers are computed to the eight closest adjacent headings in $\Theta$, and parallel maneuvers ranging from $\pm 10$ m with $1$ m resolution. The final motion primitive set $\mathcal{P}$ consists of 1184 motion primitives. More details of the lattice-based planner is found in~\cite{bergman2019improved}.   

%To speed up the A$^{\star}$ graph search used by the lattice-based planner online to solve motion planning problems, a precomputed free-space heuristic lookup-table (HLUT)~\citep{knepper2006high} is used as a well-informed heuristic function. For more details, the reader is referred to~\citep{knepper2006high}. 
\vspace*{-0.5em}
\subsection{Simulation results}
\vspace*{-0.5em}
The proposed optimization-based RHP approach is evaluated on a reverse parking scenario (see Fig.~\ref{fig:rv}) and a parallel parking scenario (see Fig~\ref{fig:pp}). The obstacles and vehicle
bodies are described by bounding circles~\citep{lavalle2006planning}. In all simulations, the time between two consecutive RHP iterations is $\delta=0.5$ s. During the simulations, it is assumed that a trajectory-tracking controller is used to follow the computed trajectories with high accuracy between each RHP iteration, however the controller design is out of the scope in this work.

The results for the reverse parking scenario are presented in Fig.~\ref{fig:rv} and Table~\ref{tab:rev_park}. As shown in Fig.~\ref{fig:rp_res}, the average difference in objective function value $\Delta J_{\text{tot}}$ increases as the planning horizon grows. The maximum achievable improvement is 26.5\% compared to the nominal solution computed by the lattice-based planner. However, extending the planning horizon beyond $T=60$ s only leads to a minor improvement. More precisely, if the full horizon (FH) in~\eqref{eq:cctoc} is improved in a single iteration as done in~\cite{bergman2019bimproved} (i.e. not using a \emph{receding} horizon approach), only an additional improvement of $3.5$\% is obtained.  
Furthermore, the average computation time for one RHP iteration $\bar{t}_{\text{RHP}}$ grows with longer planning horizon (especially for \mbox{$T>100$ s}), which is mainly due to increased problem dimension of the resulting NLP. Since the time needed to execute the trajectory is included in the cost function~\eqref{eq:cost}, a practically relevant performance measure is the total time to reach the terminal state $t_{\text{tot}}$, which is the computation time before trajectory execution can start, i.e. the latency time, plus the trajectory execution time. When the nominal solution is improved using the RHP algorithm, the additional latency time $\Delta t_{\text{lat}}$ depends only on the computation time for the first RHP iteration, since the remaining improvements are done during execution. In Table~\ref{tab:rev_park}, it is shown that the average difference in total time $\Delta \bar{t}_{\text{tot}}$ between using and not using the RHP algorithm obtains its minimum at a planning horizon of \mbox{$60-80$ s} in this scenario. Using a planning horizon in this interval, the vehicle will in average reach the terminal state more than \mbox{$30$ s} faster (latency time + motion execution time) than if the nominal trajectory is planned and executed without improvement. 

\begin{table}[t!]
	\caption{\small Summary of results from the reverse parking scenario in Fig.\ref{fig:rv}. See Fig.~\ref{fig:rv} and Fig.~\ref{fig:pp} for a description of the variables.    } \label{tab:rev_park}	
	\normalsize
	\centering
	\begin{tabular}{ccccccc}	
		$T$ [s]   & 20 & 40 & 60 & 80  & 120 & FH   \\
		\hline
		$\Delta J_{\mathrm{tot}}$ [\%] & -12.5 & -14.4 & -23.0 & -24.1 & -26.3 & -26.5 \\
		$\bar{t}_{\text{RHP}}$ [s] & 0.05 & 0.14 & 0.32 & 0.64 & 3.6 & 14.0 \\
		$\Delta \bar{t}_{\text{lat}}$ [s] & 0.34 & 0.96 & 1.8 & 3.1 & 7.3 & 14.0 \\
		$\Delta \bar{t}_{\text{tot}}$ [s] & -16.6 & -23.5 & -30.2 & -30.3 & -26.2 & -22.8 \\
		\hline
	\end{tabular}
\end{table}

The results for the parallel parking scenario (Fig.~\ref{fig:pp} and Table~\ref{tab:par_park}) are similar to the ones for the reverse parking scenario. The main differences are that the average decrease in total time $\Delta \bar{t}_{\text{tot}}$ and total objective function value $\Delta J_{\text{tot}}$ are even more significant in this scenario, with a maximum objective function value improvement of more than 40\%. The reason for this is because the lattice-based planner computes a nominal trajectory that is further away from a locally optimal solution due to the confined environment, which leaves large possibilities for improvement to the RHP algorithm. One illustrative example of this is shown in Fig.~\ref{fig:alphas}, where it can be seen that the terminal time is nearly halved compared to the nominal solution. Moreover, as can be seen in Table~\ref{tab:par_park} and Fig.~\ref{fig:pp_res}, also in this example $\Delta \bar{t}_{\text{tot}}$ and $\Delta J_{\text{tot}}$ are decreasing rapidly with increased planning horizon until \mbox{$T=60-80$ s}. Beyond that, only a minor additional decrease in $\Delta J_{\text{tot}}$ is obtained (full horizon: 2.9\%), whereas $\Delta \bar{t}_{\text{tot}}$ starts to increase due to an increased average computation time of the first RHP iteration. As a result, in this scenario the vehicle will in average reach the terminal state \mbox{$54$ s} faster (latency time + motion execution time) using the proposed RHP approach with planning horizon of \mbox{$T=80$ s} compared to when the nominal trajectory is planned and executed. 
\vspace*{-0.7em}
\begin{table}[b!]
	\caption{\small Summary of results from the parallel parking scenario in Fig.\ref{fig:pp}. See Fig.~\ref{fig:rv} and Fig.~\ref{fig:pp} for a description of the variables.    } \label{tab:par_park}	
	\normalsize
	\centering
	\begin{tabular}{ccccccc}	
		$T$ [s]  & 20 & 40 & 60 & 80  & 120 & FH   \\
		\hline
		$\Delta J_{\mathrm{tot}}$ [\%] & -24.9 & -35.2 & -40.8 & -41.7 & -43.4 & -43.7 \\
		$\bar{t}_{\text{RHP}}$ [s] & 0.09 & 0.29  & 0.77 & 2.0 & 10.4 & 17.0 \\
		$\Delta \bar{t}_{\text{lat}}$ [s] & 0.35 & 0.73  & 2.0 & 3.5 & 12.3 & 17.0 \\
		$\Delta \bar{t}_{\text{tot}}$ [s] & -32.6 & -45.7 & -53.7 & -54.0 & -45.9 & -44.1 \\
		\hline
	\end{tabular}
\end{table}
\begin{figure}[b!]
	\setlength\figureheight{0.15\textwidth}
	\setlength\figurewidth{0.4\textwidth}
	\input{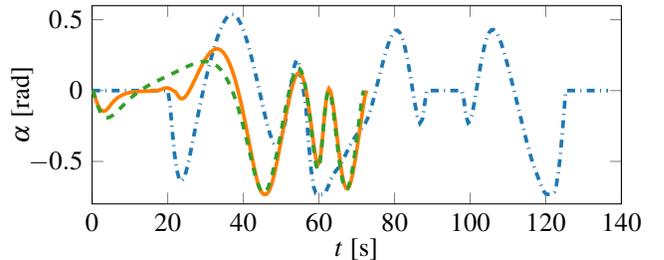}
	\vspace{-2em}
	\caption{\small The resulting steering angle trajectories for the highlighted example in Fig~\ref{fig:pp_ill}.} \label{fig:alphas}
\end{figure}
\begin{figure*}[t!] 
	\hspace{3.45em} \vspace{-2.5em}\definecolor{mycolor1}{rgb}{0.6510,0.8078,0.8902}%
\definecolor{mycolor2}{rgb}{0.1216,0.4706,0.7059}%
\definecolor{mycolor3}{rgb}{0.6980,0.8745,0.5412}%
\definecolor{mycolor4}{rgb}{0.2000,0.6275,0.1725}%
\definecolor{mycolor5}{rgb}{0.9843,0.6039,0.6000}%
\definecolor{mycolor6}{rgb}{0.8902,0.1020,0.1098}%
\definecolor{mycolor7}{rgb}{ 0.9922,0.7490,0.4353}%
\definecolor{mycolor8}{rgb}{1.0000,0.4980,0}%
\definecolor{mycolor9}{rgb}{ 0.7922,0.6980,0.8392}%
\definecolor{mycolor10}{rgb}{0.4157,0.2392,0.6039}%
\begin{tikzpicture}

\begin{axis}[%
hide axis,
width=0,
height=0,
at={(0,0)},
scale only axis,
xmin=0,
xmax=1,
ymin=0,
ymax=1,
axis background/.style={fill=white},
xmajorgrids,
ymajorgrids,
legend style={legend cell align=left,column sep=3pt, align=center},
legend columns=-1,
]
\addplot [color=mycolor2, dashdotted, line width=1.4pt]
table[row sep=crcr]{%
	1	0.0\\
};
\addlegendentry{ {\small Nominal  } }

\addplot [color=mycolor4, dashed, line width=1.4pt]
table[row sep=crcr]{%
	1	0\\
};
\addlegendentry{ {\small Full Horizon (FH) } }

\addplot [color=mycolor8, line width=1.4pt]
table[row sep=crcr]{%
	1	0\\
};
\addlegendentry{ {\small $ T = 60 $ s} }

\end{axis}
\end{tikzpicture}%
 \\
	\subfloat[][\small{Parallel parking scenario}]{
		\setlength\figureheight{0.1388\textwidth}
		\setlength\figurewidth{0.45\textwidth}
		\input{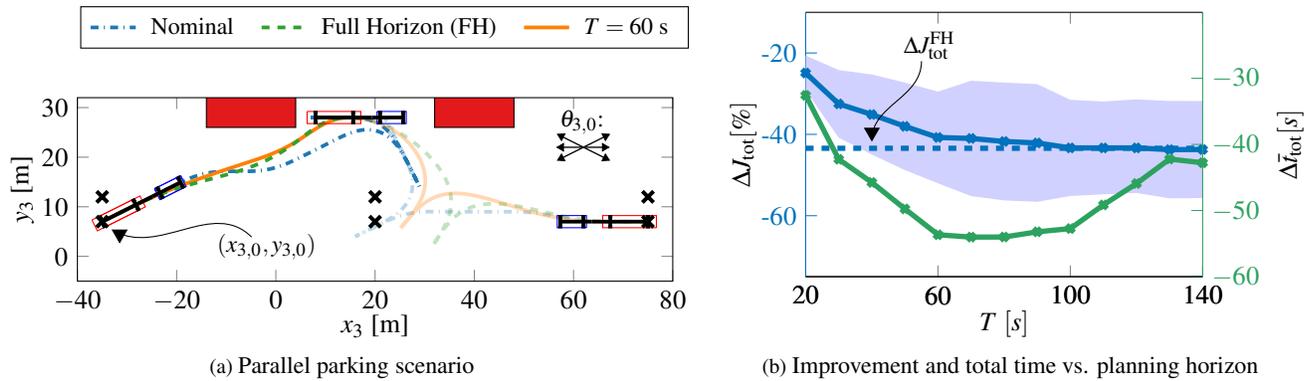} 
		\label{fig:pp_ill}
	} \subfloat[][\small Improvement and total time vs. planning horizon] {
	\setlength\figureheight{0.2\textwidth}
	\setlength\figurewidth{0.3\textwidth}
\definecolor{mycolor1}{rgb}{0.00000,0.44700,0.74100}%
\definecolor{mycolor2}{rgb}{0.85000,0.32500,0.09800}%
\definecolor{mycolor3}{RGB}{44,162,95}%
\begin{tikzpicture}

\begin{axis}[%
width=\figurewidth,
height=\figureheight,
at={(0\figurewidth,0\figureheight)},
scale only axis,
xmin=20,
xmax=140,
separate axis lines,
every outer y axis line/.append style={mycolor1},
every y tick label/.append style={font=\color{mycolor1}},
every y tick/.append style={mycolor1},
ymin=25,
ymax=90,
ytick={40, 60,  80, 100},
xtick={20, 60, 100, 140},
yticklabels={ -60, -40,  -20, 0},
%axis y discontinuity=crunch, enlargelimits=false,
%axis x line*=bottom,
axis y line*=left,
ylabel style={yshift=-0.4cm}, %shifting the y line text
xlabel style={yshift=0.2cm}, %shifting the y line text
ylabel={$\Delta J_{\text{tot}} [\%] $},
xlabel={$T$ $[s]$},
axis background/.style={fill=white}
]
\addplot[area legend, draw=none, fill opacity=0.9, fill=white!80!blue, forget plot]
table[row sep=crcr] {%
x	y\\
20	79.3800982322184\\
30	75.8067879101174\\
40	74.7207893012617\\
50	72.720186239242\\
60	70.5102892826623\\
70	73.1754449389417\\
80	72.736592593335\\
90	72.3581087807566\\
100	68.4953671259477\\
110	68.0830244167977\\
120	68.5986096114777\\
130	68.208628326412\\
140	68.1972017100503\\
150	68.2127530323575\\
150	44.4987749280598\\
140	44.3479708006752\\
130	44.272317768565\\
120	45.5977162357434\\
110	45.2525837672998\\
100	44.8472139526506\\
90	43.4154377996382\\
80	43.7672249631692\\
70	44.8197872321803\\
60	47.9622156269202\\
50	51.2135145563169\\
40	55.0340889841326\\
30	59.1696461764247\\
20	70.7859476335606\\
20	79.3800982322184\\
}--cycle;
\addplot [color=mycolor1, line width=2.0pt,mark=x, forget plot]
  table[row sep=crcr]{%
20	75.0830229328895\\
30	67.488217043271\\
40	64.8774391426972\\
50	61.9668503977795\\
60	59.2362524547912\\
70	58.997616085561\\
80	58.2519087782521\\
90	57.8867732901974\\
100	56.6712905392991\\
110	56.6678040920488\\
120	56.66981629236106\\
130	56.2404730474885\\
140	56.2725862553628\\
150	56.3557639802086\\
};
\addplot [color=mycolor1, dashed, line width = 2.0pt, forget plot]
  table[row sep=crcr]{%
20	56.5760675552101\\
150	56.5760675552101\\
};
\end{axis}

\begin{axis}[%
width=\figurewidth,
height=\figureheight,
at={(0\figurewidth,0\figureheight)},
scale only axis,
xmin=20,
xmax=140,
xtick style={draw=none},
every outer y axis line/.append style={mycolor3},
every y tick label/.append style={font=\color{mycolor3}},
every y tick/.append style={mycolor3},
ymin=-60,
ymax=-20,
ytick={-60, -50, -40, -30},
xtick={20, 60, 100, 140},
ylabel style={yshift=0.9cm}, %shifting the y line text
xlabel style={yshift=0.2cm}, %shifting the y line text
ylabel={$\Delta \bar{t}_{\text{tot}} [s] $},
axis x line*=bottom,
axis y line*=right, ylabel near ticks,
]
\addplot [color=mycolor3, line width=2.0pt, mark=x, mark options={solid, mycolor3}, forget plot]
table[row sep=crcr]{%
	20	-32.5838958322188\\
	30	-42.2519197389345\\
	40	-45.7543600621934\\
	50	-49.7538209326634\\
	60	-53.6648560208764\\
	70	-54.0189387574597\\
	80	-54.0063499307638\\
	90	-53.2238469892953\\
	100	-52.7290112560475\\
	110	-49.1488299548069\\
	120	-45.9155887125402\\
	130	-42.2116500268729\\
	140	-42.7788451495267\\
	150	-42.0885214327421\\
};
\node[anchor=east] (t) at (axis cs:68, -25){\small{$\Delta J^{\text{FH}}_{\text{tot}}$}};
\node (s2) at (axis cs:55,-26){};
\node (p2) at (axis cs:40,-41){};
\draw[->,-triangle 60] (s2) to [out=-90,in=90] (p2);
\end{axis}
\end{tikzpicture}%
\vspace{-10pt}\label{fig:pp_res}
} \\
\vspace{-7pt}
\caption{\small{(a): Parallel parking scenario from 36 different initial states. The nominal solution (dashdotted) is compared with the paths after applying the RHP algorithm using $T=60$ (solid) and the path using full horizon improvement (dashed). (b): The average difference in objective function value $\Delta J_{\text{tot}}$, and the average difference in total time $\Delta\bar{t}_{\text{tot}}$, i.e., trajectory execution time + computation time for the first RHP iteration, using different planning horizons $T$ in Algorithm~\ref{alg:rhi}. The shaded area represents $\pm$ one standard deviation.}} \label{fig:pp}
\vspace{-4pt} 
\end{figure*}

\section{Conclusions and Future Work} \label{sec:conc}

This paper introduces a new two-step trajectory planning algorithm built on a combination of a search-based motion planning algorithm and an optimization-based receding horizon planning (RHP) algorithm. While the motion planning algorithm quickly can compute a feasible, but often suboptimal, solution taking combinatorial aspects of the problem into account, the RHP algorithm based on direct optimal control techniques iteratively improves the solution quality towards the one typically achieved using direct optimal control. The receding horizon setup makes it possible for the user to conveniently trade off solution time and latency against solution quality. By exploiting the nominal dynamically feasible trajectory, a terminal manifold and a cost-to-go estimate are obtained, which make it possible to provide theoretical guarantees on recursive feasibility, non-increasing objective function value and convergence to the terminal state. These guarantees and the performance of the proposed method are successfully verified in a set of challenging trajectory planning problems for a truck and trailer system, where the proposed method is shown to significantly improve the nominal solution already for short receding planning horizons. 

Future work includes to modify the proposed receding horizon planner such that it can be applied in dynamic environments. Another extension is to improve real-time performance by using ideas from fast MPC.  

\section{Acknowledgments}
This work was partially supported by FFI/VINNOVA and the Wallenberg Artificial Intelligence, Autonomous Systems and Software Program (WASP) funded by Knut and Alice Wallenberg Foundation.
\newpage
\bibliography{myrefs.bib}		
\end{document}